\documentclass[12pt]{amsart}

\usepackage{}
\newtheorem{theorem}{Theorem}
\newtheorem{lemma}{Lemma}
\newtheorem{corollary}{Corollary}
\newtheorem{prop} {Proposition}

\theoremstyle{definition}
\newtheorem{definition}{Definition}

\theoremstyle{remark}
\newtheorem{remark}{Remark}

\newcommand{\C}{\mathbb C}
\newcommand{\hatC}{\widehat{\C}}
\numberwithin{equation}{section}
\newcommand{\T}{Teich\-m\"ul\-ler}

\usepackage{tikz}
\usetikzlibrary{matrix,arrows}
\usepackage[all]{xy}
\usepackage{pdfpages}

\makeatletter
\@namedef{subjclassname@2020}{%
	\textup{2020} Mathematics Subject Classification}
\makeatother

\begin{document}
	
	\title[Carath\'eodory metric]{Carath\'eodory metric on some generalized Teichm\"uller spaces}
	
	\author{Xinlong Dong}
	\address[Dong] {Department of Mathematics and Computer Science,\\
		Kingsborough Community College,\\
		City University of New York,\\
		Brooklyn, NY 11235--2398, USA}
	
	\email{Xinlong.Dong@kbcc.cuny.edu}
	
	\author{Sudeb Mitra}
	\address[Mitra]{Department of Mathematics,\\
		Queens College of the City University of New York,\\
		Flushing, NY 11367--1597, USA\\
		and\\
		Department of Mathematics\\
		Graduate Center of the City University of New York
		New York, NY 10016, USA}
	
	\email{sudeb.mitra@qc.cuny.edu}
	
	\dedicatory{In memory of Professor Clifford J.~Earle}

	\begin{abstract}
		We study the Carath\'eodory metric on some generalized Teichm\"uller spaces. Our paper is especially inspired by the papers \cite{E1} and \cite{Miyachi}. In \cite{E1}, Earle showed that the Carath\'eodory metric is complete on any Teichm\"uller space. In \cite{Miyachi}, Miyachi extended this result for Asymptotic Teichm\"uller spaces. We study the completeness of the Carath\'eodory metric on product Teichm\"uller spaces and on the Teichm\"uller space of a closed set in the Riemann sphere. 
	\end{abstract}
	
	\maketitle
	
	\section{Introduction}~\label{section:intro}
	
	The study of Kobayashi and Carath\'eodory metrics on {\T} spaces is an important topic. An important theorem of Royden states that the {\T} and Kobayashi metrics coincide for finite dimensional {\T} spaces; see \cite{R1}. Royden's theorem was extended to all {\T} spaces by Gardiner; see Chapter 14 of \cite{G}. Subsequently, using holomorphic motions, an easy proof was given in the paper \cite{EKK}.

	\medskip
	
	The question of Carath\'eodory metric on {\T} spaces was studied in the important paper \cite{E1}. In that paper, using Bers embedding, Earle showed that the Carath\'eodory metric is complete on {\T} spaces. In that same paper, Earle asked the question whether the Carathéodory metric coincides with the {\T} metric on {\T} spaces. In the fundamental paper \cite{Mar}, Markovi\'c proved that for any closed surface of genus $g \geq 2$, the answer is negative. 
	
	\medskip
	
	In the paper \cite{Miyachi}, Miyachi extended Earle's result to Asymptotic {\T} spaces. 
	
	\medskip
	
	Some other important papers on Kobayashi and Carath\'eodory metrics and their relationship with {\T} theory are \cite{Kra}, \cite{K1}, \cite{K2}, \cite{R1}, and \cite{Shiga}.  Other comprehensive papers on Schwarz's lemma and Kobayashi and Carath\'eodory pseudometrics are \cite{EHaHuM} and \cite{Ha}. 
	
	\medskip
	
	Our present paper is particularly inspired by the techniques used in the papers \cite{E1} and \cite{Miyachi}. We prove that, for a large class of {\T} spaces, that is, the {\it product Teichm\"uller space}, and the {\it Teichm\"uller space of a closed set in the sphere}, the Carath\'eodory metric is complete. These {\T} spaces were first studied by Lieb in his Cornell University doctoral dissertation (see \cite{Li}). Subsequently, they have been extensively studied and used in several papers. They have an intimate relationship with holomorphic motions, tame quasiconformal motions, and with some problems in geometric function theory.  For applications to holomorphic motions, see the papers \cite{M1}, \cite{M2}, \cite{MS}, and also the expository paper \cite{JM1}. For applications to some problems in geometric function theory, see the paper \cite{EM}. For applications to continuous motions and geometric function theory, see the paper \cite{JM2}. A recent application to tame quasiconformal motions is the paper \cite{JMSW}. 
	
	\medskip
	
	Our paper is arranged as follows. In \S2, we state the two main theorems of our paper. In \S3, we summarize the definitions and some important properties of Kobayashi and Carath\'eodory metrics on complex manifolds. In \S4, we give the precise definition of product {\T} spaces and note some properties that will be useful in our paper. In \S5, we prove the first main theorem of our paper. The crucial step is Theorem 1 (in \S5), where we prove an estimate for Carath\'eodory and Kobayashi metrics on {\T} spaces. In \S6, we define the {\T} space of a closed set in the Riemann sphere, and note some properties that are relevant to our paper. In \S7, we prove the second main theorem of our paper. 
	
	\bigskip
	
	{\bf Acknowledgement.} We thank the referee for his/her careful reading of the paper and for his/her helpful comments.

	\section{Statements of the main theorems}
	
	Throughout this paper, we will use $\C$ for the complex plane , $\hatC = \C \cup \{\infty\}$ for the Riemann sphere, and $\Delta = \{z \in \C: \vert z\vert < 1\}$ for the open unit disk. 
	
	\medskip
	
	We state the main theorems of this paper. 
	
	\medskip
	
	For each $i$ in the index set $I$, let $X_i$ be a hyperbolic Riemann surface. Let $X$ be the disjoint union $\coprod_{i \in I} X_i$, and let $Teich(X)$ denote its {\it product Teichm\"uller space}; the precise definition is given in \S4.2.
	
	\medskip
	
	{\bf Theorem A.} The Carath\'eodory metric on $Teich(X)$ is complete. 
	
	\medskip
	
	Let $E$ be a closed subset of $\hatC$ that contains the points $0$, $1$, and $\infty$. Let $T(E)$ denote its {\T} space; see \S6.1 for the precise definition. 
	
	\medskip
	
	{\bf Theorem B.} The Carath\'eodory metric on $T(E)$ is complete. 
	
	\section{Kobayashi and Carath\'eodory metrics}
	
	In this section, we summarize the definitions and some basic properties of the Kobayashi and Carath\'eodory pseudometrics. Let $\rho$ denote the Poincar\'e metric on $\Delta$. We have:
	$$\rho(z_1, z_2) = \tanh^{-1}\Big | \frac{z_1 - z_2}{1 - \bar{z_1}z_2}\Big |.$$

	\subsection{Kobayashi pseudometric}
	
	Let $M$ be a complex manifold. The {\it Kobayashi pseudometric} $K_M$ is defined as follows: Given two points $p$, $q$ $\in M$, we choose points $p = p_0, p_1, \cdot \cdot \cdot, p_{k-1}, p_k = q$ of $M$, points $a_1, \cdot \cdot \cdot, a_k, b_1, \cdot \cdot \cdot, b_k$ of $\Delta$, and holomorphic maps $f_1, \cdot \cdot \cdot, f_k$ of $\Delta$ into $M$ such that $f_i(a_i) = p_{i-1}$ and $f_i(b_i) = p_i$ for $i = 1, \cdot \cdot \cdot, k$. For each choice of points and maps thus made, consider the number
	$$\rho(a_1, b_1) + \cdot \cdot \cdot + \rho(a_k, b_k).$$
	Then, $K_M(p, q)$ is the infimum of the numbers obtained in this manner for all possible choices. 
	
	\medskip
	
	We note some important properties of $K_M$. For proofs, refer to Chapter 4, Section 1 of \cite{Kob}. 
	
	\begin{prop}\label{prop1}
		Let $M$ and $N$ be two complex manifolds and let $f: M \to N$ be a holomorphic map. Then
		$$K_M(p, q) \geq K_N(f(p), f(q)) \qquad \mbox { for } p, q \in M.$$
	\end{prop}
	
	\begin{corollary}\label{cor1}
		Every biholomorphic map $f: M \to N$ is an isometry; which means:
		$$K_M(p,q) = K_N(f(p), f(q)) \qquad \mbox { for } p, q \in M.$$
	\end{corollary}
	
	\begin{prop}\label{prop2}
		For the open unit disk $\Delta$, $K_{\Delta}$ coincides with the Poincar\'e metric $\rho$.
	\end{prop}
	
	\medskip
	
	The following fact is stated in \S3 of \cite{Ha}. For the sake of completeness, we include a proof here. 
	
	\begin{prop} \label{prop3}
		Let $B_r(a)$ be the open ball of radius $r$ and center $a$ in a complex Banach space $X$. Then
		$$K_{B_r(a)}(a, x) = \tanh^{-1}\Big(\frac{\Vert x-a\Vert}{r}\Big)$$
		for all $x$ in $B_r(a)$. 
	\end{prop}
	
	\begin{proof}
		Let $x \in B_r(a)$. We may assume that $x \not= a$.  Define a holomorphic map $f: \Delta \to B_r(a)$ as follows:
		$$f(t) = a + \frac{rt(x-a)}{\Vert x-a\Vert}.$$
		Note that 
		$$f\Big(\frac{\Vert x-a\Vert}{r}\Big) = x.$$ 
		By Proposition 1, we have 
		$$K_{B_r(a)}\Big(f(0), f(\frac{\Vert x-a\Vert}{r})\Big) \leq K_{\Delta}\Big(0, \frac{\Vert x-a\Vert}{r}\Big)$$
		which gives
		\begin{equation}\label{upperbound}
			K_{B_r(a)}(a, x) \leq \tanh^{-1}\Big(\frac{\Vert x-a\Vert}{r}\Big).
		\end{equation}
		Next, by Hahn-Banach theorem there exists a linear functional $\ell \in X^{*}$ with $\Vert\ell\Vert = 1$ and $\ell(x-a) = \Vert x-a\Vert$. Define a holomorphic map $g: B_{r}(a) \to \Delta$ given by
		$$g(y) = \frac{\ell(y-a)}{r}.$$
		Again, by Proposition 1, we have 
		$$K_{\Delta}(g(x), g(a)) \leq  K_{B_r(a)}(x,a).$$
		It immediately follows that 
		\begin{equation}\label{lowerbound}
			K_{B_r(a)}(a, x) \geq \tanh^{-1}\Big(\frac{\Vert x-a\Vert}{r}\Big).
		\end{equation}
		Combining inequalities (\ref{upperbound}) and (\ref{lowerbound}) we get the required result. 
	\end{proof}

	\subsection{Carath\'eodory pseudometric}
	
	Let $M$ be a complex manifold. The {\it Carath\'eodory pseudometric} $C_M$ is defined as follows:
	$$C_M(p, q) = \sup_{f}\rho(f(p), f(q)) \qquad \mbox { for } p, q \in M,$$
	where the supremum is taken with respect to the family of holomorphic maps $f: M \to \Delta$.
	
	\medskip
	
	We note some important properties of $C_M$. For proofs, refer to Chapter 4, Section 2 of \cite{Kob}.
	\begin{prop}\label{prop4}
		Let $M$ and $N$ be two complex manifolds and let $f: M \to N$ be a holomorphic map. Then
		$$C_M(p, q) \geq C_N(f(p), f(q)) \qquad \mbox { for } p, q \in M.$$
	\end{prop}
	
	\begin{corollary}\label{cor2}
		Every biholomorphic map $f: M \to N$ is an isometry; which means:
		$$C_M(p,q) = C_N(f(p), f(q)) \qquad \mbox { for } p, q \in M.$$
	\end{corollary}
	
	\begin{prop}\label{prop5}
		For the open unit disk $\Delta$, $C_{\Delta}$ coincides with the Poincar\'e metric $\rho$.
	\end{prop}
	
	\begin{prop}\label{prop6}
		If $M$ and $M'$ are complex manifolds with complete Carath\'eodory metric, so is $M \times M'$.
	\end{prop}
	
	\begin{prop}\label{prop7}
		Let $B_r(a)$ be the open ball of radius $r$ and center $a$ in a complex Banach space $X$. Then
		$$C_{B_r(a)}(a, x) = \tanh^{-1}\Big(\frac{\Vert x-a\Vert}{r}\Big)$$
		for all $x$ in $B_r(a)$. 
	\end{prop}
	
	See Lemma 2 in \cite{E1}. The proof is similar to the proof of Proposition \ref{prop3}.
	
	\begin{corollary}\label{cor3}
		The Carath\'eodory metric induces the standard topology on $B_{r}(a)$.
	\end{corollary}
	
	See Corollary of Lemma 2 in \cite{E1}.

	\section{Some properties of product {\T} spaces}
	
	We study some basic properties of product {\T} spaces. The details are given in Sections 7.1 to 7.8 in \cite{EM}. For standard facts on classical {\T} spaces, the reader is referred to the standard references \cite{E2}, \cite{E3}, \cite{G}, \cite{GL}, \cite{Hu}, and \cite{Nag}. 
	
	\subsection{Some complex Banach spaces}
	Let $I$ be an index set. For full generality, in this section we will assume that $I$ is uncountable. For each $i$ in the index set $I$, let $X_i$ be a hyperbolic Riemann surface. Let $X$ be the disjoint union $\coprod_{i \in I} X_i$. We introduce the following important Banach spaces:
	
	\bigskip
	
	By definition, a Beltrami form on $X$ is a tensor $\mu$ whose restriction to each $X_i$ is a bounded measurable Beltrami form $\mu_i$ on $X_i$ with $L^{\infty}$ norm less than some finite constant independent of $i$ in $I$. We define
	$$\Vert\mu\Vert = \sup\{\Vert\mu_i\Vert_{\infty}: i \in I\}.$$
	We denote the Banach space of Beltrami forms on $X$ by $Belt(X)$ and we denote the open unit ball of $Belt(X)$ by $M(X)$. The basepoint of $M(X)$ is its center $0$. 
	
	\bigskip
	
	Let $\pi: \Delta \to R$ be a holomorphic universal covering of the hyperbolic Riemann surface $R$. Every holomorphic quadratic differential $\psi$ on $R$ lifts to a holomorphic quadratic differential $\widetilde\psi(z)dz^2$ on $\Delta$. We say that $\psi$ is bounded if its Nehari norm
	$$\Vert\psi\Vert_{N} = \sup\{\Vert \widetilde\psi(z)\Vert (1 - \vert z\vert^{2})^2: z \in \Delta\}$$
	is finite. 
	
	\bigskip
	
	For each $i$ in $I$ let $X_i^*$ be the conjugate Riemann surface of $X_i$, and let $X^*$ be the disjoint union of the $X_i^*$. Let $\psi$ be a holomorphic quadratic differential on $X^*$. We say that $\psi$ is bounded if its restriction $\psi_i$ to $X_i^*$ is a bounded holomorphic quadratic differential for each $i$ and its Nehari norm
	$$\Vert\psi\Vert_{N} = \sup\{\Vert\psi_i\Vert_{N}: i \in I\}$$
	is finite. We denote the complex Banach space of bounded holomorphic quadratic differentials on $X^*$ by $B(X^*)$.

	\subsection{Product {\T} space}
	For each $i \in I$, let $Teich(X_i)$ be the {\T} space of the Riemann surface $X_i$, let $0_i$ be the basepoint of $Teich(X_i)$ and let $d_i$ be the {\T} metric on $Teich(X_i)$. By definition, the {\T} space $Teich(X)$ is the set of functions $t$ on $I$ such that $t(i)$
	is in $Teich(X_i)$ for each $i$ and the set of numbers $\{d_i (0_i, t(i)): i \in I\}$ is bounded. As usual, we shall write $t_i$ for $t(i)$. The basepoint of $Teich(X)$ is the function $t$ such that $t_i= 0_i$ for each $i$; we shall denote it by $0_{X}$.
	
	The {\T} metric on $Teich(X)$ is defined by
	$$d_T(s, t) = \sup\{d_i(s_i, t_i): i \in I\},$$
	for $s$ and $t$ in $Teich(X)$. 
	Since each metric $d_i$ is complete, the metric $d_{T}$ on $Teich(X)$ is also complete.
	
	\begin{lemma}\label{lem1}
		The {\T} metric on $Teich(X)$ is the same as its Kobayashi metric.
	\end{lemma}
	
	See Proposition 7.28 in \cite{EM}.
	
	\medskip

	For each $i \in I$, let $\Phi_i$ be the usual projection of $M(X_i)$ onto $Teich(X_i)$; see, for example, \cite{E2}, \cite{E3}, \cite{Hu}, \cite{Nag} for standard facts on the classical {\T} spaces. By definition $d_i(0_i, t_i) = \inf\{\rho(0, \Vert\mu_i\Vert): \mu_i \in M(X_i) \mbox { and } \Phi_i(\mu_i) = t_i\}$ for each $t \in Teich(X)$, so if $\mu \in M(X)$ then $d_i(0_i, \Phi(\mu_i)) \leq \rho(0, \Vert\mu\Vert)$ for all $i$. We can therefore define the {\it standard projection} $\Phi: M(X) \to Teich(X)$ by the formula
	$$\Phi(\mu)_{i} = \Phi_i(\mu_i), \qquad \mu \in M(X) \mbox { and } i \in I.$$
	It is easy to see that the map $\Phi$ is surjective. 
	
	\begin{definition}\label{def1}
		For each $i \in I$ let $\mathcal B_i: M(X_i) \to B(X_i^*)$ be the classical Bers projection (see \cite{E2}, \cite{E3}, \cite{Hu}, or \cite{Nag}).  The {\it generalized Bers projection} $\mathcal B: M(X) \to B(X^*)$ is defined by the formula $\mathcal B(\mu)_{i} = \mathcal B_i(\mu_i)$, $i$ in $I$ and $\mu$ in $M(X)$. 
	\end{definition}
	
	\begin{definition}\label{def2}
		For each $i \in I$, let $\alpha_i: B(X_i^*) \to L^{\infty}(X_i)$ be the classical Ahlfors-Weill map (see \cite{E1}, \cite{E2}, \cite{GL}, \cite{Hu}, or \cite{Nag}). The {\it generalized Ahlfors-Weill map} $\alpha: B(X^*) \to Belt(X)$ is defined by the formula $\alpha (\psi)_{i} = \alpha_i(\psi_i)$, $i$ in $I$ and $\psi$ in $B(X^*)$. 
	\end{definition}
	
	\begin{prop}\label{prop8}
		The generalized Bers projection $\mathcal B: M(X) \to B(X^*)$ is a holomorphic split submersion with the following properties:
		
		\smallskip
		
		(i) $\mathcal B(0) = 0$ and $\Vert\mathcal B(\mu)\Vert_{N} \leq 6$ for all $\mu$ in $M(X)$;
		
		(ii) for all $\mu$ and $\nu$ in $M(X)$, $\mathcal B(\mu) = \mathcal B(\nu)$ if and only if $\Phi(\mu) = \Phi(\nu)$;
		
		(iii) if $\psi \in B(X^*)$ and $\Vert\psi\Vert_{B(X^*)} < 2$, then $\mathcal B(\alpha(\psi)) = \psi$. 
		
	\end{prop}
	
	Statements (i), (ii), and (iii) follow immediately from the corresponding statements in the classical case (see \cite{E2}, \cite{E3}, \cite{GL}, \cite{Hu}, or \cite{Nag}). The fact that $\mathcal B$ is a holomorphic split submersion is proved in Proposition 7.3 in \cite{EM}.
	
	\begin{corollary}\label{cor4}
		There is a unique complex Banach manifold structure on $Teich(X)$ that has the following properties:
		
		\smallskip
		
		(i) the map $\Phi: M(X) \to Teich(X)$ is a holomorphic split submersion;
		
		(ii) the map $\widehat{\mathcal B}: Teich(X) \to \mathcal{B}(M(X))$ such that $\widehat{\mathcal B} \circ \Phi = \mathcal B$ is biholomorphic,
		
		(iii) if $t \in Teich(X)$ and $\Vert\widehat{\mathcal B}(t)\Vert_{B(X^*)} < 2$, then $\Phi(\alpha(\widehat{\mathcal B}(t))) = t$. 
	\end{corollary}
	
	See Corollary 7.4 in \cite{EM}. 
	
	\begin{definition}\label{def3}
		The biholomorphic map $\widehat{\mathcal B}$ is called the {\it generalized Bers embedding} of $Teich(X)$ in $B(X^*)$.
	\end{definition}

	\begin{definition}\label{def4}
		The {\it generalized Ahlfors-Weill section of $\mathcal B$} is the restriction of the map $\alpha$ to the set of $\psi$ in $B(X^*)$ with $\Vert\psi\Vert_{B(X^*)} < 2$. 
		
	\end{definition}
	
	\subsection{Changing the basepoint}
	For each $i \in I$, let $h_i$ be a $K$-quasiconformal mapping of $X_i$ onto a hyperbolic Riemann surface $Y_i$, with $K$ independent of $i$. Let $Y = \coprod Y_i$ be the disjoint union.
	Each $h_i$ induces a biholomorphic map $h_i^{*}$ of $Teich(X_i)$ onto $Teich(Y_i)$. 
	
	\begin{prop}\label{prop9}
		There is a unique biholomorphic map $h$ of $Teich(X)$ onto $Teich(Y)$ such that  $h^{*}(t)_i = h_i^{*}(t_i)$ for all $t$ in $Teich(X)$ and $i$ in $I$. Furthermore, if $\mu$ is the point in $M(X)$ such that $\mu_i$ is the Beltrami coefficient of $h_i$, then $h^{*}$ maps the point $\Phi(\mu)$ in $Teich(X)$ to the basepoint $0_Y$ of $Teich(Y)$; here $\Phi: M(X) \to Teich(X)$ is the standard projection.
	\end{prop}
	
	For a proof, see Proposition 7.9 in \cite{EM}.

	\begin{remark}\label{rmk1}
		For any $\mu$ in $M(X)$ and any $i$ in $I$ there are a Riemann surface $Y_i$ and a quasiconformal mapping of $X_i$ onto $Y_i$ whose Beltrami coefficient is $\mu_i$. Therefore each point $\Phi(\mu)$ in $Teich(X)$ can be mapped to the basepoint of some $Teich(Y)$ by some biholomorphic map $h^{*}$. 
	\end{remark}
	
	\section{Proof of Theorem A}
	For each $i$ in the index set $I$, let $X_i$ be a hyperbolic Riemann surface. Let $X$ be the disjoint union $\coprod_{i \in I} X_i$. Let $Teich(X)$ be the product {\T} space discussed in \S4. Let $C_T$ and $K_T$ respectively denote the Carath\'eodory and Kobayashi metrics on $Teich(X)$.
	
	\medskip
	
	Let $0$ denote the origin of the complex Banach space $B(X^*)$ in \S4.1. To simplify notations, let $B_2(0)$ denote the ball of radius 2 centered at the origin in $B(X^*)$, and let $B_6(0)$ denote the ball of radius 6 centered at the origin in $B(X^*)$. Let $C_{B_6(0)}$ denote the Carath\'eodory metric on $B_6(0)$, and let $K_{B_2(0)}$ denote the Kobayashi metric on $B_2(0)$.

	\begin{theorem}\label{thm1}
		$\tanh C_T(x,y) \leq \tanh K_T(x,y) \leq 3\tanh C_T(x,y)$ for all $x$, $y$ in $Teich(X)$.
	\end{theorem}
	
	\begin{proof}  \underline{Step 1.}  Let $x = 0$; to simplify notations, we use $0$ for the basepoint of $Teich(X)$.
		
		We know that 
		$$C_T(0,t) \leq K_T(0,t) \qquad \mbox { for all } t \in Teich(X)$$
		and so we have $\tanh C_T(0,t) \leq \tanh K_T(0,t)$ for all $t$ in $Teich(X)$. 
		
		\smallskip
		
		By the generalized Bers embedding discussed in Proposition \ref{prop8} and Corollary \ref{cor4}, we know that $B_2(0) \subset Teich(X) \subset B_6(0)$. Hence we have $C_T(0,t) \geq C_{B_6(0)}(0,t)$ and therefore,
		$$C_T(0,t) \geq \tanh^{-1} \frac{\Vert t\Vert}{6} \qquad \mbox { for all } t \in Teich(X).$$
		
		It follows that
		\begin{equation}\label{lowerbound2}
			6\tanh C_T(0, t) \geq \Vert t\Vert
		\end{equation} 
		for all $t \in Teich(X)$. 
		
		\smallskip
		
		We also have $K_T(0,t) \leq K_{B_2(0)}(0,t)$ which implies that 
		$$K_T(0, t) \leq \tanh ^{-1}\frac{\Vert t\Vert}{2} \qquad \mbox { if } \Vert t\Vert < 2.$$
		It follows that
		$$2\tanh K_T(0, t) \leq \Vert t\Vert \qquad \mbox { if } \Vert t\Vert < 2.$$
		
		This is also true if $\Vert t\Vert > 2$. It follows that 
		\begin{equation}\label{upperbound2}
			2\tanh K_T(0, t) \leq \Vert t\Vert
		\end{equation}
		for all $t \in Teich(X)$. 
		
		\smallskip
		
		Combining (\ref{lowerbound2}) and (\ref{upperbound2}), we get 
		$$2\tanh K_T(0,t) \leq \Vert t\Vert \leq 6\tanh C_T(0,t)$$
		for all $t \in Teich(X)$; hence,
		$$\tanh K_T(0,t) \leq 3\tanh C_T(0,t)$$
		for all $t \in Teich(X)$. 
		
		\medskip
		
		\underline{Step 2.} Let $x$, $y$ be any points in $Teich(X)$. By Remark \ref{rmk1}, there exists a biholomorphic map $h^{*}$ on $Teich(X)$ onto some {\T} space $Teich(Y)$ such that $h^{*}$ maps the point $x$ in $Teich(X)$ to the basepoint $0_Y$ of $Teich(Y)$. Let $h^{*}(y) = t \in Teich(Y)$. 
		
		By Step 1, we have 
		$$\tanh C_T(0_Y, t) \leq \tanh K_T(0_Y, t) \leq 3 \tanh C_T(0_Y, t) \mbox { for all } t \in Teich(Y),$$
		where $C_T$ denotes the Carathéodory metric on $Teich(Y)$ and $K_T$ denotes the Kobayashi metric on $Teich(Y)$.
		
		Since $h^{*}$ is biholomorphic, the Kobayashi and Carathéodory metrics on $Teich(X)$ and $Teich(Y)$ respectively, are preserved. It follows that 
		$$\tanh C_T(x,y) \leq \tanh K_T(x,y) \leq 3\tanh C_T(x,y)$$ 
		for all $x$, $y$ in $Teich(X)$.
	\end{proof}

	{\bf Proof of Theorem A.}  Let $\{t_n\}$ be a Cauchy sequence in $Teich(X)$, with respect to the Carathéodory metric $C_T$. 
	
	Let $\epsilon > 0$ be given. Choose 
	$$\widehat\epsilon = \tanh^{-1}\Big(\frac{1}{3}\tanh\epsilon\Big).$$
	
	Then, for this $\widehat\epsilon > 0$, there exists a positive integer $N$ such that for all $m, n > N$, we have $C_T(t_m, t_n) < \widehat\epsilon$. Hence, $3 \tanh C_T(t_m, t_n) < 3 \tanh \widehat\epsilon$ for all $m, n > N$. It follows from Theorem \ref{thm1} that 
	$$\tanh K_T(t_m, t_n) < 3 \tanh \widehat\epsilon = 3\big(\frac{1}{3}\tanh\epsilon\big) = \tanh\epsilon$$
	for all $m , n > N$. 
	
	Therefore, $K_T(t_m, t_n) < \epsilon$ for all $m, n > N$. 
	Hence, $\{t_n\}$ is a Cauchy sequence with respect to $K_T$. By Lemma \ref{lem1}, it follows that $\{t_n\}$ is a Cauchy sequence with respect to the {\T} metric $d_T$. Since $d_T$ is complete,  $t_n \to t$ in $Teich(X)$, and by Lemma \ref{lem1} again, $t_n \to t$ in $Teich(X)$ with respect to $K_T$, and so, we have $K_T(t_n, t) \to 0$. 
	
	\smallskip
	
	Let $\widetilde\epsilon > 0$ be given. There exists a natural number $\widetilde N> 0$ such that $K_T(t_n, t) < \widetilde\epsilon$ for all $n > \widetilde N$. Hence, $\tanh K_T(t_n, t) < \tanh \widetilde\epsilon$ for all $n > \widetilde N$. It follows by Theorem \ref{thm1} that
	$\tanh C_T(t_n, t)  < \tanh \widetilde\epsilon$ for all $n > \widetilde N$. Therefore, $C_T(t_n, t) < \widetilde\epsilon$ for all $n > \widetilde N$. It follows that $t_n \to t$ in $Teich(X)$ with respect to $C_T$. Hence, the Carathéodory metric on $Teich(X)$ is complete. \qed
	
	\begin{remark}\label{rmk2}
		Let $S$ be a hyperbolic Riemann surface, and let $Teich(S)$ denote its usual {\T} space. In \cite{E1}, Earle proved that the Carathéodory metric on Teich(S) is complete. Our method gives an alternative proof of Earle's theorem; in particular the estimate on the right-hand side of Theorem \ref{thm1} is explicit. 
	\end{remark}
	
	\begin{remark}\label{rmk3}
		If the index set is finite, then $Teich(X)$ is simply the (finite) cartesian product of the {\T} spaces $Teich(X_i)$. In this case, Theorem A follows from the main theorem in Earle's paper \cite{E1} and Proposition 6.
		
	\end{remark}

	\section{Teichm\"uller space of a closed set in the Riemann sphere}
	
	Recall that a homeomorphism of $\hatC$ is called {\it normalized} if it fixes the points $0$, $1$, and $\infty$. We use $M(\C)$ to denote the open unit ball of the complex Banach space $L^{\infty}(\C)$. Each $\mu$ in $M(\C)$ is the Beltrami coefficient of a unique normalized quasiconformal homeomorphism $w^{\mu}$ of $\hatC$ onto itself. The basepoint of $M(\C)$ is the zero function. 
	
	\smallskip
	
	The Kobayashi metric $K_{M(\C)}$ on $M(\C)$ is defined by
	
	$$K_{M(\C)} (\mu, \nu) = \tanh^{-1}\Vert(\mu - \nu)(1 - \overline{\mu}\nu)^{-1}\Vert_{\infty}$$
	for all $\mu$, $\nu$ in $M(\C)$.
	
	\subsection{Teichm\"uller space of a closed set in the Riemann sphere} Let $E$ be a closed subset of $\hatC$ that contains the points $0$, $1$, and $\infty$.
	
	\begin{definition}\label{def5}
		Two normalized quasiconformal self-mappings $f$ and $g$ of $\hatC$ are said to be $E$-equivalent if and only if $f^{-1} \circ g$ is isotopic to the identity rel $E$. The {\it Teichm\"uller space $T(E)$} is the set of all $E$-equivalence classes of normalized quasiconformal self-mappings of $\hatC$.  The basepoint of $T(E)$ is the $E$-equivalence class of the identity map. 
	\end{definition}
	
	We can define the quotient map 
	$$P_E: M(\C) \to T(E)$$
	by setting $P_E(\mu)$ equal to the $E$-equivalence class of $w^{\mu}$, written as $[w^{\mu}]_E$. Clearly, $P_E$ maps the basepoint of $M(\C)$ to the basepoint of $T(E)$. 
	
	\smallskip
	
	The {\T} metric $d_{T(E)}$ on $T(E)$ is given by
	$$d_{T(E)}(P_E(\mu), t) = \inf \{K_{M(\C)}(\mu, \nu): \nu \in M(\C) \mbox { and } P_E(\nu) = t\}$$
	for all $\mu$ in $M(\C)$ and $t$ in $T(E)$.

	Since $E^{c}$ is an open subset of $\C \setminus \{0,1\}$, each of its connected components is a hyperbolic Riemann surface. We index these components $X_i$ by a set $I$ of positive integers, and we form the product {\T} space of their disjoint union $E^{c}$; let $Teich(E^{c})$ denote this product {\T} space. Let $M(E)$ be the open unit ball in $L^{\infty}(E)$. Then, the product $Teich(E^{c}) \times M(E)$ is a complex Banach manifold. 
	
	\smallskip
	
	For $\mu$ in $L^{\infty}(\C)$, let $\mu|E^{c}$ and $\mu|E$ be the restrictions of $\mu$ to $E^c$ and $E$ respectively. We define the projection map $\widetilde P_E$ from $M(\C)$ to $Teich(E^c) \times M(E)$ by the formula
	
	\begin{equation}\label{projectionmap}
		\widetilde P_E(\mu) = \big(\Phi(\mu|E^{c}), \mu|E\big)
	\end{equation}
	for all $\mu$ in $M(\C)$, where $\Phi: M(E^{c}) \to Teich(E^{c})$ is the standard projection.
	
	\medskip
	
	We now state ``Lieb's isomorphism theorem."
	
	\begin{theorem}\label{thm10}
		For all $\mu$ and $\nu$ in $M(\C)$ we have $P_E(\mu) = P_E(\nu)$ if and only if $\widetilde P_E(\mu) = \widetilde P_E(\nu)$. Consequently, there is a well defined bijection $\theta: T(E) \to Teich(E^{c}) \times M(E)$ such that $\theta \circ P_E = \widetilde P_E$, and $T(E)$ has a unique complex manifold structure such that $P_E$ is a holomorphic split submersion and the map $\theta$ is biholomorphic.
	\end{theorem}
	
	See \S7.10 in \cite{EM} for a complete proof. 
	
	\begin{prop}\label{prop11}
		The {\T} and Kobayashi metrics on $T(E)$ are equal.
	\end{prop}
	
	See Proposition 7.30 in \cite{EM}.

	\section{Proof of Theorem B}

	By Theorem A, the Carathéodory metric on $Teich(E^c)$ is complete, and it is well-known that the Carathéodory metric on $M(E)$ is complete. Therefore, by Proposition 6, $Teich(E^c) \times M(E)$ is also complete. It follows by Theorem 2, that the Carathéodory metric on $T(E)$ is complete. \qed
	
	\begin{remark}
		If $E$ is a finite set, then $T(E)$ is naturally identified with the classical {\T} space $Teich(\hatC \setminus E)$. This easily follows from 
		Theorem 2. A direct proof is given in Example 3.1 of \cite{M1}. Therefore, when $E$ is finite, Theorem B is exactly the main theorem in Earle's paper \cite{E1}. 
	\end{remark}
	
	\begin{remark}
		If $E = \hatC$, then $T(E)$ is naturally identified with $M(\C)$. In this case, the Carath\'eodory, Kobayashi, and Poincar\'e metrics coincide.
	\end{remark}
	
	\begin{remark}
		If the closed set $E$ has zero area, then $M(E)$ contains only one point, and $Teich(E^c) \times M(E)$ is isomorphic to $Teich(E^c)$ and in that case, Theorem B is a special case of Theorem A. 
		
	\end{remark}
	
	\bibliographystyle{amsalpha}

\end{document}